\documentclass[portuges,12pt,letter]{article}
\usepackage[centertags]{amsmath}
\usepackage{amsfonts}
\usepackage{newlfont}
\usepackage{amscd}
\usepackage{graphics}
\usepackage{epsfig}
\usepackage{indentfirst}
\usepackage{amsxtra}
\usepackage[latin1]{inputenc}
\usepackage{amssymb, amsmath}
\usepackage{amsthm}
\usepackage[mathscr]{eucal}

\newtheorem{thm}{Theorem}[section]

\newtheorem{remark}[thm]{Remark}

\author{Fabio Silva Botelho \\ Departamento de Matemática \\ Universidade Federal de Santa Catarina, UFSC \\
Florian\'{o}polis, SC - Brazil}

\title{\bf  A duality principle for non-convex optimization in $\mathbb{R}^n$} 

\begin{document}
\maketitle

\abstract{This article develops a duality principle for a class of optimization problems in $\mathbb{R}^n$.
The results are obtained based on standard tools of convex analysis and on a well known result of Toland for D.C. optimization.
 Global sufficient optimality conditions are also presented as well as  relations between the critical points of the
 primal and dual formulations. Finally we formally prove there is no duality gap between the primal and dual formulations in a local extremal context.}

\section{Introduction} Consider a function $J:\mathbb{R}^n \rightarrow \mathbb{R}$ defined by
$$J(x)=-G_1(x)+G_2(x,\mathbf{0}),$$
where
$$G_1(x)=-\frac{x^TAx}{2}+\frac{K}{2} x^T x-f^Tx$$
and
$$G_2(x,v)=\sum_{j=1}^N \frac{\gamma_j}{2} \left(\frac{x^T B_jx}{2}+c_j +v_j\right)^2+\frac{K}{2} x^Tx,$$
and where
$x \in \mathbb{R}^n$, $v \in \mathbb{R}^N$, $A$ is a $n \times n$ real symmetric matrix, $B_j$ is a $n \times n$ real symmetric matrix and $c_j,\gamma_j \in \mathbb{R}$, where $\gamma_j>0,$
$\forall j \in \{1,\ldots,N\}.$

Finally, $f \in \mathbb{R}^n$ as well.

Observe that $$J(x)=\frac{x^TAx}{2}+\sum_{j=1}^N \frac{\gamma_j}{2} \left(\frac{x^T B_jx}{2}+c_j\right)^2+f^Tx.$$

We shall develop a duality principle which has no restriction concerning $n$ and $N$, so that it includes the case $n \neq N.$

Also, we establish a relation between the corresponding critical points of the primal and dual formulations.

The main result is established through an extension of a Toland result found in \cite{12}.

Indeed, we must emphasize our work is a kind of extension and continuation of the original works of Bielski and Telega \cite{2900,85} combined with the work of Toland \cite{12}. The technical details follow in some extent the results in \cite{12a}. Anyway, we highlight once more our work in some sense complements the  results
in \cite{2900,85} but now applied to a $\mathbb{R}^n$ simpler context.

Similar problems have been addressed in \cite{17,9}, among others.

\section{The main result}

We start this section with a remark.

\begin{remark} About the notation we denote the canonical basis of $\mathbb{R}^n$ by $$\{e_1, \ldots, e_n\}$$ and we recall that in general $A^T$ denotes the transpose of the matrix $A$. For a  $n \times n$ matrix $A$ we denote $A>\mathbf{0}$ if $A$ is positive definite. Finally, $I_d$ denotes the identity matrix $n \times n$ and by $\{\delta_{ij}\}$ we denote the standard  $N \times N$ Kronecker delta, that is,
\begin{equation} \delta_{ij}=\left\{
\begin{array}{lr}
 1, & \text{ if } i=j
 \\
 0, & \text{ otherwise},  \end{array} \right.\end{equation}
 $\forall i,j \in \{1,\ldots,N\}.$
\end{remark}

Our main result is summarized by the following theorem.

\begin{thm} Let $J:\mathbb{R}^n \rightarrow \mathbb{R}$ be defined by
\begin{eqnarray}J(x)&=&\frac{x^TAx}{2}+\sum_{j=1}^N \frac{\gamma_j}{2} \left(\frac{x^T B_j x}{2}+c_j\right)^2+f^Tx
\nonumber \\ &=& -G_1(x)+G_2(x,\mathbf{0})\end{eqnarray}
where
$$G_1(x)=-\frac{x^TAx}{2}+\frac{K}{2} x^T x-f^Tx$$
and
$$G_2(x,v)=\sum_{j=1}^N \frac{\gamma_j}{2} \left(\frac{x^T B_jx}{2}+c_j +v_j\right)^2+\frac{K}{2} x^Tx.$$

Assume $A$ is a $n \times n$ symmetric matrix and $B_j$ are $n \times n$  symmetric  matrices $\forall j \in \{1,\dots,N\}$ such that $$J(x) \rightarrow +\infty$$ as $|x| \rightarrow \infty,$ and
$K>0$ is such that $K I_d > A$.

Define also $G^*_1:\mathbb{R}^n \rightarrow \mathbb{R}$ by
\begin{eqnarray}
G_1^*(v^*)&=&\sup_{x \in \mathbb{R}^n} \{(v^*)^Tx-G_1(x)\}
\nonumber \\ &=& \frac{1}{2}(v^*+f)^T (K I_d-A)^{-1}(v^*+f)\end{eqnarray}
and
$G_2^*: \mathbb{R}^n \times C^* \rightarrow \mathbb{R}$ by
\begin{eqnarray}
G_2^*(v^*,v_0^*)&=& \sup_{(x,v) \in \mathbb{R}^n\times \mathbb{R}^N}\{(v^*)^T x+(v_0^*)^Tv-G_2(x,v)\} \nonumber \\
&=& \frac{1}{2}(v^*)^T\left(\sum_{j=1}^N (v_0^*)_j B_j+K I_d\right)^{-1} v^*+\sum_{j=1}^N \frac{1}{2\gamma_j} (v_0^*)_j^2
\nonumber \\ &&-\sum_{j=1}^Nc_j(v_0^*)_j
\end{eqnarray}  where
$$C^*=\left\{v_0^* \in \mathbb{R}^N\;:\; \sum_{j=1}^N (v_0^*)_jB_j+K I_d>\mathbf{0}\right\}.$$
Moreover,  define
$$B^*=\left\{v_0^* \in \mathbb{R}^N\;:\; A+\sum_{j=1}^N (v_0^*)_jB_j>\mathbf{0}\right\}$$
and
$$A^*=B^* \cap C^*.$$

At this point we denote $$J^*(v^*,v_0^*)=G_1(v^*)-G_2^*(v^*,v_0^*),$$ and define
$$\tilde{J}^*(v^*)=\sup_{v_0^* \in C^*} J^*(v^*,v_0^*).$$

Assume $x_0 \in \mathbb{R}^n$ is such that $\delta J(x_0)=\mathbf{0}$ and
define
$$(\hat{v}_0^*)_j=\gamma_j\left(\frac{x_0^TB_jx_0}{2}+c_j\right),$$

$$\hat{v}^*=\sum_{j=1}^N (\hat{v}_0^*)_j B_j x_0+K x_0,$$
$$H_3=P_1\;\overline{E}\;P_2,$$
$$\alpha\equiv(\alpha)_{n \times n}=(I_d-H_3)D-I_d,$$
and $$\alpha_1= -\left(\sum_{p=1}^N (\hat{v}_0^*)_p B_p+K I_d\right)^{-1}(\alpha)\left(\sum_{p=1}^N (\hat{v}_0^*)_p B_p+K I_d\right).$$
where  \begin{equation} P_1=\left[
\begin{array}{cccc}
 B_1x_0 & B_2 x_0 & \cdots & B_Nx_0  \end{array} \right]_{n \times N}\end{equation}
and
\begin{equation} P_2=\left[
\begin{array}{c}
 x_0^T B_1\left(\sum_{p=1}^N (\hat{v}_0^*)_p B_p+K I_d\right)^{-1}\\
 x_0^T B_2\left(\sum_{p=1}^N (\hat{v}_0^*)_p B_p+K I_d\right)^{-1}\\
 \vdots \\
 x_0^T B_N\left(\sum_{p=1}^N (\hat{v}_0^*)_p B_p+K I_d\right)^{-1}   \end{array} \right]_{N \times n}\end{equation}
where
$$E= \{E_{l\eta}\}=\left[ \gamma_l \left(x_0^T B_l \left(\sum_{p=1}^N (\hat{v}_0^*)_p B_p+K I_d\right)^{-1} B_\eta x_0\right)+\delta_{l\eta}\right]_{N \times N}$$
and $$\overline{E}=\{\overline{E}_{l\eta}\}=\{E_{l\eta}\}^{-1}.$$

Furthermore,
$$D=  \hat{B} \left(\sum_{p=1}^N (\hat{v}_0^*)_p B_p+K I_d\right)^{-1}+I_d$$ where $$\hat{B}_{n\times n}=\{\hat{B}_{jk}\}=\left\{ \sum_{l=1}^N\sum_{s,q=1}^n \gamma_l\;(x_0)_s(B_l)_{js} (B_l)_{qk}(x_0)_q\right\}.$$

Under such assumptions and notation, we have,
\begin{enumerate}
\item If $\delta^2J(x_0)>\mathbf{0}$, $\delta^2 J(x_0)+(K I_d-A)(\alpha_1)> \mathbf{0}$ and $\hat{v}_0^* \in C^*$, then
$$\delta \tilde{J}(\hat{v}^*)=\mathbf{0},$$ and
$$\delta^2 \tilde{J}(\hat{v}^*)>\mathbf{0},$$ so that there exist $r>0$ and $r_1>0$ such that
\begin{eqnarray}
J(x_0)&=& \inf_{x \in B_r(x_0)} J(x) \nonumber \\ &=& \inf_{v^* \in B_{r_1}(\hat{v}^*)} \tilde{J}^*(v^*)\nonumber \\ &=&
\tilde{J}^*(\hat{v}^*)\nonumber \\ &=& \inf_{v^* \in B_{r_1}(\hat{v}^*)} \sup_{v^*_0 \in C^*} J^*(v^*,v_0^*)
\nonumber \\ &=& J^*(\hat{v}^*,\hat{v}_0^*).
\end{eqnarray}
\item If $\hat{v}_0^* \in A^*$ so that $$\delta^2J(x_0)>\mathbf{0},$$ define
$$J^*_2(v^*)=\sup_{ v_0^* \in A^*} J^*(v^*,v_0^*).$$
Thus in such a case, we have
$$\delta J_2^*(\hat{v}^*)=\mathbf{0},$$ $$\delta^2 J_2^*(\hat{v}^*)> \mathbf{0}$$ and
\begin{eqnarray}
J(x_0)&=& \inf_{x \in \mathbb{R}^n} J(x) \nonumber \\ &=& \inf_{v^* \in \mathbb{R}^n}J_2^*(v^*)\nonumber \\ &=&
J_2^*(\hat{v}^*)\nonumber \\ &=& \inf_{v^* \in \mathbb{R}^n}\sup_{v^*_0 \in A^*} J^*(v^*,v_0^*)
\nonumber \\ &=& J^*(\hat{v}^*,\hat{v}_0^*).
\end{eqnarray}
\item If  $\delta^2J(x_0)< \mathbf{0}$, $\delta^2 J(x_0)+(K I_d-A)(\alpha_1)<\mathbf{0}$ and $\hat{v}_0^* \in C^*$ then
$$\delta \tilde{J}(\hat{v}^*)=\mathbf{0}$$ and $$\delta^2\tilde{J}^*(\hat{v}^*) < \mathbf{0},$$ so that
there exist $r>0$ and $r_1>0$ such that
\begin{eqnarray}
J(x_0)&=& \sup_{x \in B_r(x_0)} J(x) \nonumber \\ &=& \sup_{v^* \in B_{r_1}(\hat{v}^*)} \tilde{J}^*(v^*)\nonumber \\ &=&
\tilde{J}^*(\hat{v}^*)\nonumber \\ &=& \sup_{v^* \in B_{r_1}(\hat{v}^*)}\sup_{v^*_0 \in C^*} J^*(v^*,v_0^*)
\nonumber \\ &=& J^*(\hat{v}^*,\hat{v}_0^*).
\end{eqnarray}

\end{enumerate}
\end{thm}
\begin{proof} From $\delta J(x_0)=\mathbf{0}$ we obtain
$$Ax_0+\sum_{j=1}^N \gamma_j\left(\frac{x_0^TB_j x_0}{2}+c_j\right) B_j x_0+f=\mathbf{0}.$$

Hence
\begin{eqnarray}\label{r0}-Ax_0+K x_0 -f &=& \sum_{j=1}^N \gamma_j\left(\frac{x_0^TB_j x_0}{2}+c_j\right) B_j x_0 +K x_0 \nonumber \\ &=&
\sum_{j=1}^N (\hat{v}_0^*)_jB_j x_0+K x_0 \nonumber \\ &=& \hat{v}^*.
\end{eqnarray}
Thus,
$$x_0=(K I_d-A)^{-1}(\hat{v}^*+f),$$ so that
\begin{eqnarray}
(K I_d-A)^{-1}(\hat{v}^*+f)-\left(\sum_{j=1}^N (\hat{v}_0^*)_j B_j +K I_d\right)^{-1}\hat{v}^*&=& x_0-x_0 \nonumber \\ &=& \mathbf{0},
\end{eqnarray}
and therefore
$$\frac{\partial J^*(\hat{v}^*,\hat{v}_0^*)}{\partial v^*}=\mathbf{0}.$$

From this and and the implicit function theorem, we get
\begin{equation}\label{r1}\frac{\partial \tilde{J}^*(\hat{v}^*)}{\partial v^*}=\frac{\partial J^*(\hat{v}^*,\hat{v}_0^*)}{\partial v^*}
+\sum_{j=1}^N\frac{\partial J^*(\hat{v}^*,\hat{v}_0^*)}{\partial (v^*_0)_j}\frac{\partial (\hat{v}_0^*)_j}{\partial v^*}.\end{equation}
However, from
$$(\hat{v}_0^*)_j=\gamma_j\left(\frac{x_0^T B_j x_0}{2}+c_j\right),$$ we have
\begin{eqnarray}
0 &=&  -\frac{(\hat{v}_0^*)_j}{\gamma_j}+\frac{x_0^T B_j x_0}{2}+c_j \nonumber \\ &=& \frac{\partial J^*(\hat{v}^*,\hat{v}_0^*)}{\partial (v^*_0)_j},
\; \forall j \in \{1,\ldots,N\},
\end{eqnarray}
so that from (\ref{r1}), we obtain
\begin{equation}\frac{\partial \tilde{J}^*(\hat{v}^*)}{\partial v^*}=\frac{\partial J^*(\hat{v}^*,\hat{v}_0^*)}{\partial v^*}
= \mathbf{0}\end{equation}
Hence, we may denote
$$\delta \tilde{J}^*(\hat{v}^*)=\mathbf{0}.$$
On the other hand from (\ref{r0}), we have
\begin{eqnarray}
G_1^*(\hat{v}^*)&=&(\hat{v}^*)^T x_0-\frac{K}{2}x_0^Tx_0+\frac{1}{2} x_0^T A x_0 + f^T x_0 \nonumber \\ &=& (\hat{v}^*)^T x_0-G_1(x_0),
\end{eqnarray}
and
$$ G_2^*(\hat{v}^*,\hat{v}_0^*)= (\hat{v}^*)^Tx_0+(\hat{v}^*_0)^T \mathbf{0}-G_2(x_0,\mathbf{0}).$$
Therefore
\begin{eqnarray}
\tilde{J}^*(\hat{v}^*)&=& J^*(\hat{v}^*,\hat{v}_0^*) \nonumber \\ &=& G_1^*(\hat{v}^*)-G_2^*(\hat{v}^*,\hat{v}_0^*) \nonumber \\
&=& -G_1(x_0)+G_2(x_0,\mathbf{0}) \nonumber \\ &=& J(x_0).
\end{eqnarray}

Observe also that
\begin{eqnarray}
\delta^2 \tilde{J}^*(\hat{v}^*)&=& \left\{ \frac{\partial^2 \tilde{J}^*(\hat{v}^*)}{\partial v^*_j \partial v^*_k}\right\} \nonumber \\ &=&
\left\{ \frac{\partial^2 J^*(\hat{v}^*, \hat{v}_0^*)}{\partial v^*_j \partial v^*_k} + \sum_{l=1}^N\frac{\partial^2 J^*(\hat{v}^*, \hat{v}_0^*)}{\partial v^*_j \partial (v^*_0)_l}
\frac{\partial (\hat{v}_0^*)_l}{\partial v^*_k} \right\},\end{eqnarray}
where $\hat{v}_0^*$ is such that
\begin{eqnarray}
&&\frac{\partial J^*(\hat{v}^*,\hat{v}_0^*)}{\partial (v_0^*)_l}\nonumber \\ &=&
\frac{1}{2}(v^*)^T\left(\sum_{j=1}^N (\hat{v}_0^*)_jB_j+K I_d\right)^{-1} B_l\left(\sum_{j=1}^N (\hat{v}_0^*)_jB_j+K I_d\right)^{-1} (v^*) \nonumber \\ &&
-\frac{(v_0^*)_l}{\gamma_l}+c_l \nonumber \\ &=& 0.
\end{eqnarray}

Taking the variation of this last equation in $v^*_k$, we get
\begin{eqnarray}
&&e_k^T\left(\sum_{j=1}^N (\hat{v}_0^*)_jB_j+K I_d\right)^{-1} B_l x_0 \nonumber \\ &&- \sum_{\eta=1}^N\left(x_0^TB_l \left(\sum_{j=1}^N (\hat{v}_0^*)_jB_j+K I_d\right)^{-1} B_\eta x_0 \;  \frac{\partial (\hat{v}_0^*)_\eta}{\partial v^*_k}
\right)
\nonumber \\ && -\frac{1}{\gamma_l}\frac{\partial (\hat{v}_0^*)_l}{\partial v^*_k}\nonumber \\ &=& 0
\end{eqnarray}

From this, denoting
$$\frac{-1}{\gamma_l} \frac{\partial (\hat{v}_0^*)_l}{\partial v^*_k}=\frac{-1}{\gamma_l}
\sum_{\eta=1}^N\delta_{l\eta} \frac{\partial (\hat{v}_0^*)_\eta}{\partial v^*_k}$$ we obtain

\begin{eqnarray}
&&\left\{\frac{\partial (\hat{v}_0^*)_l}{\partial v^*_k}\right\} \nonumber \\ &=&
\left[ x_0^T B_l\left(\sum_{j=1}^N (\hat{v}_0^*)_jB_j+K I_d\right)^{-1} B_\eta x_0+\frac{1}{\gamma_l} \delta_{l\eta} \right]^{-1} \nonumber \\ && \times
\left[ x_0^T B_\eta \left(\sum_{j=1}^N (\hat{v}_0^*)_jB_j+K I_d\right)^{-1} e_k \right] \nonumber \\ &=& \overline{E} P_2. \end{eqnarray}
Also
\begin{eqnarray}&& \left\{\frac{\partial^2 J^*(\hat{v}^*,\hat{v}_0^*)}{\partial v_j^*\partial (v_0^*)_l}\right\} \nonumber \\ &=&
\left[ e_j^T \left(\sum_{p=1}^N (\hat{v}_0^*)_pB_p+K I_d\right)^{-1} B_l x_0\right]_{n \times N}
 \nonumber \\ &=&  \left(\sum_{p=1}^N (\hat{v}_0^*)_p B_p+K I_d\right)^{-1}\left[
\begin{array}{cccc}
 B_1x_0 & B_2 x_0 & \cdots & B_Nx_0  \end{array} \right]_{n \times N} \nonumber \\ &=&  \left(\sum_{p=1}^N (\hat{v}_0^*)_p B_p+K I_d\right)^{-1} P_1,
 \end{eqnarray}
 so that
 \begin{eqnarray}&&\left\{\sum_{l=1}^N\frac{\partial^2 J^*(\hat{v}^*,\hat{v}_0^*)}{\partial v_j^*\partial (v_0^*)_l} \frac{\partial (\hat{v}_0^*)_l}{\partial v^*_k}\right\}
 \nonumber\\ &=&
 \left(\sum_{p=1}^N (\hat{v}_0^*)_p B_p+K I_d\right)^{-1} P_1 \overline{E} P_2 \nonumber \\ &=&
 \left(\sum_{p=1}^N (\hat{v}_0^*)_p B_p+K I_d\right)^{-1} H_3.
 \end{eqnarray}

Therefore
\begin{eqnarray}
\delta^2 \tilde{J}^*(\hat{v}^*)&=& \left\{ \frac{\partial^2 \tilde{J}^*(\hat{v}^*)}{\partial v^*_j \partial v^*_k}\right\} \nonumber \\ &=&
\left\{ \frac{\partial^2 J^*(\hat{v}^*, \hat{v}_0^*)}{\partial v^*_j \partial v^*_k} + \sum_{l=1}^N\frac{\partial^2 J^*(\hat{v}^*, \hat{v}_0^*)}{\partial v^*_j \partial (v^*_0)_l}
\frac{\partial (\hat{v}_0^*)_l}{\partial v^*_k} \right\}
\nonumber \\ &=& -\left(\sum_{p=1}^N (\hat{v}_0^*)_p B_p+K I_d\right)^{-1}+
\left(K I_d-A\right)^{-1} \nonumber \\ &&+\left(\sum_{p=1}^N (\hat{v}_0^*)_p B_p+K I_d\right)^{-1} H_3. \end{eqnarray}
Therefore, recalling that $$D=\hat{B} \left(\sum_{p=1}^N (\hat{v}_0^*)_p B_p+K I_d\right)^{-1} +I_d$$ where $$\hat{B}_{n\times n}=\{\hat{B}_{jk}\}=\left\{ \sum_{l=1}^N\sum_{s,q=1}^n \gamma_l\;(x_0)_s(B_l)_{js} (B_l)_{qk}(x_0)_q\right\},$$ we may write
\begin{eqnarray}\label{q30}
&& \delta^2 \tilde{J}^*(\hat{v}^*)\;D  \nonumber \\ &=&  -\left(\sum_{p=1}^N (\hat{v}_0^*)_p B_p+K I_d\right)^{-1}(I_d-H_3)D\ \nonumber \\ && +
\left(K I_d-A\right)^{-1} D \nonumber \\ &=&-\left(\sum_{p=1}^N (\hat{v}_0^*)_p B_p+K I_d\right)^{-1}((I_d-H_3)D-I_d+I_d)\ \nonumber \\ && +
\left(K I_d-A\right)^{-1} D \nonumber \\ &=&-\left(\sum_{p=1}^N (\hat{v}_0^*)_p B_p+K I_d\right)^{-1}(I_d+\alpha)\ \nonumber \\ && +
\left(K I_d-A\right)^{-1} D \nonumber \\ &=&(K I_d-A)^{-1}\left(-(K I_d- A)(I_d-\alpha_1)
\right.\ \nonumber \\ && \left.+ \hat{B}+\sum_{p=1}^N (\hat{v}_0^*)_p B_p+K I_d \right)\left(\sum_{p=1}^N (\hat{v}_0^*)_p B_p+K I_d\right)^{-1}
\end{eqnarray}
 Therefore, denoting also $$H_1=(K I_d-A)^{-1},$$ $$H_2= \left(\sum_{p=1}^N (\hat{v}_0^*)_p B_p+K I_d\right)^{-1},$$ we have
\begin{eqnarray}
 && \delta^2 \tilde{J}^*(\hat{v}^*)\;D \nonumber \\ &=&H_1\left( A(I_d-\alpha_1)+\hat{B}+\sum_{\eta=1}^N (\hat{v}_0^*)_\eta B_\eta+ K I_d(\alpha_1)\right)H_2
\nonumber \\ &=& H_1(\delta^2 J(x_0)+(K I_d-A)(\alpha_1))H_2.\end{eqnarray}
Since  $D$, $H_1$ and $H_2$ are symmetric positive definite matrices, assuming  $\delta^2 J(u_0)>\mathbf{0}$ and $\delta^2 J(x_0)+(K I_d-A)(\alpha_1)>\mathbf{0}$, we have
$$\delta^2 \tilde{J}(\hat{v}^*)>\mathbf{0},$$ so that
there exist $r>0$ and $r_1>0$ such that
\begin{eqnarray}
J(x_0)&=& \inf_{x \in B_r(x_0)} J(x) \nonumber \\ &=& \inf_{v^* \in B_{r_1}(\hat{v}^*)} \tilde{J}(v^*)\nonumber \\ &=&
\tilde{J}(\hat{v}^*)\nonumber \\ &=& \inf_{v^* \in B_{r_1}(\hat{v}^*)} \sup_{v^* \in C^*} J^*(v^*,v_0^*)
\nonumber \\ &=& J^*(\hat{v}^*,\hat{v}_0^*).
\end{eqnarray}
Assume now $\hat{v}_0^* \in A^*$ so that $$\delta^2J(x_0)>\mathbf{0}.$$

Observe that if $v_0^* \in A^*$, then
$$J^*(v^*,v_0^*)=G_1^*(v^*)-G_2^*(v^*,v_0^*)$$ is such that
$$\frac{\partial J^*(v^*,v_0^*)}{\partial (v^*)^2}=(KI_d-A)^{-1}-\left(\sum_{j=1}^N (v_0^*)_jB_j+KI_d\right)^{-1}>\mathbf{0},$$
so that defining $$J^*_2(v^*)=\sup_{ v_0^* \in A^*} J^*(v^*,v_0^*)$$
 we have that $J_2^*$ is convex as the supremum of a family of convex functions.

Similarly as above, we may obtain
$$\delta J_2^*(\hat{v}^*)=\mathbf{0}$$ and $$J_2^*(\hat{v}^*)=J(x_0)=J^*(\hat{v}^*,\hat{v}_0^*).$$
From this, since $J_2^*$ is convex, from the min-max theorem and from the general result in Toland \cite{12}, we may infer that
\begin{eqnarray}
J_2^*(\hat{v}^*)&=& \inf_{ v^* \in \mathbb{R}^n} J_2^*(v^*) \nonumber \\ &=&\inf_{ v^* \in \mathbb{R}^n}\sup_{v^*_0 \in A^*} J^*(v^*,v_0^*)
\nonumber \\ &=& \sup_{v^*_0 \in A^*} \inf_{v^* \in \mathbb{R}^n} J^*(v^*,v_0^*)
\nonumber \\ &\leq& \sup_{v_0^* \in A^*} \left\{-G_1(x)+\frac{K}{2}x^Tx+\sum_{j=1}^N\left((v_0^*)_j\left(\frac{x^TB_j x}{2}+c_j\right)-  \frac{(v^*_0)_j^2}{2\gamma_j}\right)\right\}
\nonumber \\ &\leq& \sup_{v_0^* \in \mathbb{R}^N} \left\{-G_1(x)+\frac{K}{2}x^Tx+\sum_{j=1}^N\left((v_0^*)_j\left(\frac{x^TB_j x}{2}+c_j\right)-  \frac{(v^*_0)_j^2}{2\gamma_j}\right)\right\}
\nonumber \\ &=& -G_1(x)+G_2(x,\mathbf{0}) \nonumber \\ &=& J(x),\; \forall x \in \mathbb{R}^n.\end{eqnarray}
Hence
$$\inf_{x \in \mathbb{R}^n} J(x) \geq J_2^*(\hat{v}^*)=J(x_0),$$
so that
\begin{eqnarray}
J(x_0)&=& \inf_{x \in \mathbb{R}^n} J(x) \nonumber \\ &=& \inf_{v^* \in \mathbb{R}^n}J_2(v^*)\nonumber \\ &=&
J_2(\hat{v}^*)\nonumber \\ &=& \inf_{v^* \in \mathbb{R}^n}\sup_{v^*_0 \in A^*} J^*(v^*,v_0^*)
\nonumber \\ &=& J^*(\hat{v}^*,\hat{v}_0^*).
\end{eqnarray}
Finally, the proof of third item  is similar to that of the first one.

 This would complete the proof.
\end{proof}
\begin{remark} For the special case in which $n=N=1$ we obtain $\alpha_1=0.$\end{remark}
\begin{remark} We may obtain an even more interesting result if we consider a more general case in which $K$ is a symmetric matrix
$n \times n$. Specifically for the case $$K=K I_d=A+\varepsilon I_d$$ we get
$$K I_d-A=\varepsilon I_d,$$ and in such a case
\begin{eqnarray}&&\delta^2 \tilde{J}^*(\hat{v}^*)\;D \nonumber \\ &=& H_1(\delta^2 J(x_0)+(K I_d-A)(\alpha_1))H_2
\nonumber \\ &=&H_1(\delta^2 J(x_0)+\varepsilon I_d(\alpha_1))H_2 \nonumber \\ &=&
H_1(\delta^2 J(x_0)+\mathcal{O}(\varepsilon) I_d)H_2\end{eqnarray}
so that we recover at least approximately a correspondence between $\delta^2 J(x_0)$ and $\delta^2 \tilde{J}^*(\hat{v}^*),$ up to considering the
sign of $H_2$ as well.

Observe that in this last context,
$$H_1=\frac{1}{\varepsilon} I_d$$ and $$H_2= \left(A+\sum_{p=1}^N (\hat{v}_0^*)_p B_p+\varepsilon I_d\right)^{-1}.$$
\end{remark}
\begin{remark} Let us now consider a dual functional proposed in the current literature (see \cite{9}, for example).
For the model addressed in this article, such a functional is expressed as
$$-J_1^*(v_0^*)= \frac{1}{2} f^T \left( \sum_{p=1}^N (v_0^*)_p B_p+A\right)^{-1} f + \sum_{p=1}^N \frac{(v_0^*)_p^2}{2\gamma_p}-\sum_{p=1}^N c_p (v_0^*)_p.$$

Taking the variation (in fact derivative) of such a functional in $(v_0^*)_j$, since the matrices in question are symmetric,  we obtain
\begin{eqnarray}&&-\frac{\partial J_1^*(v_0^*)}{\partial (v_0^*)_j} \nonumber \\ &=&-\frac{1}{2} f^T \left( \sum_{p=1}^N (v_0^*)_p B_p+A\right)^{-1}B_j \left( \sum_{p=1}^N (v_0^*)_p B_p+A\right)^{-1} f + \frac{(v_0^*)_j}{\gamma_j}- c_j\nonumber \\ &=& -\frac{1}{2} x_0^T B_j x_0+ \frac{(v_0^*)_j}{\gamma_j}- c_j.\end{eqnarray}

Now taking the  derivative of this expression relating $(v_0^*)_k$ we get
\begin{eqnarray}&&\left\{-\frac{\partial^2 J_1^*(v_0^*)}{\partial (v_0^*)_j \partial (v_0^*)_k}\right\}=  \left\{f^T \left( \sum_{j=1}^N (v_0^*)_p B_p+A\right)^{-1}B_j \left( \sum_{j=1}^N (v_0^*)_p B_p+A\right)^{-1} \right. \nonumber \\ && \left.\times B_k \left( \sum_{j=1}^N (v_0^*)_p B_p+A\right)^{-1} f + \frac{\delta_{jk}}{\gamma_j}\right\}.\end{eqnarray}
Since the matrices in question are symmetric, at a critical point  as specified in the last theorem, we obtain,
\begin{eqnarray}&&\left\{-\frac{\partial^2 J_1^*(\hat{v}_0^*)}{\partial (v_0^*)_j \partial (v_0^*)_k} \right\}\nonumber \\ &=& \left\{ x_0^T B_j\left( \sum_{p=1}^N (\hat{v}_0^*)_p B_p+A\right)^{-1}B_k  x_0 + \frac{\delta_{jk}}{\gamma_j}\right\}.\end{eqnarray}

On the other hand, for the functional $J(x)$ we obtain
\begin{equation}\delta^2 J(x_0)=A+ \hat{B}+\sum_{p=1}^N(\hat{v}_0^*)_pB_p \end{equation}

where $$\hat{B}=\hat{B}_{n\times n}=\{\hat{B}_{jk}\}=\left\{ \sum_{l=1}^N\sum_{s,q=1}^n \gamma_l\;(x_0)_s(B_l)_{js} (B_l)_{qk}(x_0)_q\right\}.$$

From this we may see that there exists a qualitative correspondence (in terms of positivity or negativity in a matrix sense)  between the two second derivative matrices only for the special case $n=N=1$. Even so we have to consider the sign of $\sum_{p=1}^N (\hat{v}_0^*)_p B_p+A$ to get a right conclusion.

For a general case such a correspondence may not hold even if $n=N.$

\end{remark}

\section{Conclusion}
In this article we have developed a duality principle for a class of non-convex optimization problems in $\mathbb{R}^n$.
For such a class of problems we address the case in which for the variables in question, $n \neq N.$

We believe to have obtained a very interesting way of developing the dual formulation, establishing a correct relation
between the critical points of the primal and dual problems, with no duality gap between such primal and dual formulations.

This problem has been addressed in similar form in \cite{ 17,9}, for example. It is not our objective here to
comment extensively such previous results, but just offer a new possibility of obtaining the dual formulations for such a class of  problems.

\end{document}